\documentclass{amsart}

\usepackage{tikz-cd}

\usepackage{amssymb}
\usepackage{amsthm}
\usepackage{amscd}

\usepackage[mathscr]{eucal}
\usepackage{enumitem}

\usepackage{hyperref}

\usepackage{color}

\newtheorem{Theorem}{Theorem}[section]
\newtheorem{theorem}[Theorem]{Theorem}

\newtheorem{proposition}[Theorem]{Proposition}
\newtheorem{prop}[Theorem]{Proposition}

\newtheorem{corollary}[Theorem]{Corollary}

\newtheorem{lemma}[Theorem]{Lemma}
\newtheorem{Fact}[Theorem]{Fact}
\newtheorem{fact}[Theorem]{Fact}
\newtheorem{remark/def}[Theorem]{Remark/Definition}

\newtheorem{claim}[Theorem]{Claim}

\newtheorem*{maincor}{Corollary \ref{cor:char of KP and SH in Sigma}}

\theoremstyle{definition}

\newtheorem{example}[Theorem]{Example}

\newtheorem{remark}[Theorem]{Remark}

\newtheorem{definition}[Theorem]{Definition}

\newtheorem{question}[Theorem]{Question}
\newtheorem{def/rem}[Theorem]{Definition/Remark}
\newtheorem{def/fact}[Theorem]{Definition/Fact}
\newtheorem{not/rem}[Theorem]{Notation/Remark}

\newsavebox{\indbin}
\savebox{\indbin}{\begin{picture}(0,0)
\newlength{\gnu}
\settowidth{\gnu}{$\smile$} \setlength{\unitlength}{.5\gnu}
\put(-1,-.65){$\smile$} \put(-.25,.1){$|$}
\end{picture}}

\newcommand{\be}{\begin{enumerate}}
\newcommand{\bi}{\begin{itemize}}

\newcommand{\bd}{\begin{defn}}
\newcommand{\bs}{\boldsymbol}
\newcommand{\bt}{\begin{theorem}}
\newcommand{\bl}{\begin{lemma}}
\newcommand{\ee}{\end{enumerate}}
\newcommand{\ei}{\end{itemize}}
\newcommand{\ed}{\end{defn}}
\newcommand{\et}{\end{theorem}}
\newcommand{\el}{\end{lemma}}

\newcommand{\id}{\operatorname{id}}

\newcommand{\aut}{\operatorname{Aut}}
\newcommand{\aute}{\operatorname{Aut}_{\bs{e}}}

\newcommand{\autf}{\operatorname{Autf}}

\newcommand{\LL}{\operatorname{L}}
\newcommand{\gal}{\operatorname{Gal}}
\newcommand{\gall}{\gal_{\LL}}
\newcommand{\KP}{\operatorname{KP}}
\newcommand{\SH}{\operatorname{S}}

\newcommand{\galsh}{\gal_{\SH}}

\newcommand{\galkp}{\gal_{\KP}}

\newcommand{\autfl}{\operatorname{Autf_{\LL}}}

\newcommand{\autflres}{\operatorname{Autf_{\LL}^{res}}}

\newcommand{\autfkp}{\operatorname{Autf_{\KP}}}

\newcommand{\autfsh}{\operatorname{Autf_{\SH}}}

\newcommand{\lequiv}{\equiv^{\LL}}
\newcommand{\kpequiv}{\equiv^{\KP}}
\newcommand{\sequiv}{\equiv^{\operatorname{S}}}
\newcommand{\shequiv}{\equiv^{\operatorname{S}}}

\def\res{\operatorname{res}}

\def\dcl{\operatorname{dcl}}

\def\bdd{\operatorname{bdd}}
\def\acl{\operatorname{acl}}

\def\tp{\operatorname{tp}}

\begin{document}


\title{Relativized Galois groups of first order theories over a hyperimaginary}

\date{\today}

\author[H. Lee]{Hyoyoon Lee$^{\dagger}$}
\address{$^\dagger$Department of Mathematics, Yonsei University\\ 50 Yonsei-Ro, Seodaemun-Gu\\ Seoul, 03722, South Korea\\ ORCiD:0000-0002-4745-8663}
\email{hyoyoonlee@yonsei.ac.kr}

\author[J. LEE]{Junguk Lee$^{\ddagger}$ $^{\ast}$}
\address{$^{\ddagger}$Department of Mathematics, Changwon National University\\ 20 Changwondaehak-ro, Uichang-gu\\ Changwon-si, Gyeongsangnam-do, 51140, South Korea\\ ORCiD: 0000-0002-2150-6145}
\email{ljwhayo@changwon.ac.kr}

 \thanks{$^{\dagger}$The first author is supported by NRF of Korea grants 2021R1A2C1009639, the Yonsei University Research Fund(Post Doc. Researcher Supporting Program) of 2023 (project no.: 2023-12-0159), and the Research Institute for Mathematical Sciences, an International Joint Usage/Research Center located in Kyoto University. $^{\ddagger}$The second author is supported by Basic Science Research Program through the National Research Foundation of Korea (NRF) funded by the Ministry of Education(2022R1I1A1A0105346711) and supported by the new faculty Financial Program at Changwon National University in 2023. $^{\ast}$Corresponding author.\\
\indent The authors thank the anonymous referee for careful reading and all the comments and suggestions. Especially, Example \ref{example:necessary condition on H}  and the closedness of $\aut_{\equiv^H}(\mathfrak{C},\bs{e})$ in Remark \ref{rem:res_aut_ptwise_top} are suggested by the referee.  
}

\begin{abstract}
We study relativized Lascar groups, which are formed by relativizing Lascar groups to the solution set of a partial type $\Sigma$. We introduce the notion of a Lascar tuple for $\Sigma$ and
by considering the space of types over a Lascar tuple for $\Sigma$, the topology for a relativized Lascar group is (re-)defined and some fundamental facts about the Galois groups of first-order theories are generalized to the relativized context.

In particular, we prove that any closed subgroup of a relativized Lascar group corresponds to a stabilizer of a bounded hyperimaginary having at least one representative in the solution set of the given partial type $\Sigma$.
Using this, we find the correspondence between subgroups of the relativized Lascar group and the relativized strong types.
\end{abstract}

\maketitle

D. Lascar introduced the notion of Lascar (Galois) group of a first-order theory in his paper \cite{L82}, which is proven to have Galois theoretic correspondence between the subgroups of the automorphism group and the notions of model theoretic strong types (cf.  \cite{CLPZ01, LP01}).
That is, the Lascar, KP, and Shelah strong types, which play crucial roles in the study of first-order theories, are exactly the the same as the orbit equivalence relations defined by the corresponding subgroups of the Lascar group.

In \cite{DKL17}, the relativized Lascar group, which are formed by restricting the domain of the automorphisms that fix the Lascar strong types in a parital type, was introduced to study the first homology groups in model theory.
And then the relativized Lascar group was studied further in \cite{DKKL21}, proving that under some conditions, the relativized Lascar groups of Shelah strong types of some finite tuples are all isomorphic.
In \cite{KL23}, the Galois theoretic correspondence for the Lascar group and the results concerning the relativized Lascar group in the previous papers were extended and generalized to the context of hyperimaginaries (= $\emptyset$-type-definable equivalence classes).
There are also studies connecting the Lascar group and topological dynamics (such as \cite{KPR18}), but this paper does not delve into that topic.

Our contribution is centered around the extension of the correspondence between subgroups of the Lascar group and the strong types over hyperimaginaries in \cite{KL23} to the relativized context.
Because we prove every statement in this paper over a fixed hyperimaginary $\bs{e}$, we often encounter more technical difficulties;
we tried our best to prove and explain every such detail as much as we could, and sometimes separated them as remarks and claims to emphasize that hyperimaginaries can cause subtleties.

There are two natural ways to relativize a subgroup of automorphisms to a partial type $\Sigma$:
One is simply by restricting the domains of the automorphisms in the subgroup directly (Definition \ref{def:restricted groups}, Definition \ref{def:general definition of subgroup}(2)), and the other one is collecting restricted automorphisms that fixes the corresponding orbits in the partial type (Definition \ref{def:relativized Galois groups}, Definition \ref{def:general definition of subgroup}(1)).
We will compare them in Section \ref{sec:rest and example}, but until then, we will focus on the latter definition, re-define the topology of the relativized Lascar group, and find the correspondence between the subgroups and the strong types in the given partial type.

Given a hyperimaginary $\bs{e}$, by assuming that a type $\Sigma$ is $\bs{e}$-invariant and there is a small tuple $c$ of realizations of $\Sigma$ such that $\bs{e} \in \dcl(c)$, we successfully generalize classical results to the relativized context and prove following Corollary \ref{cor:char of KP and SH in Sigma}, which says that having the same strong type in the solution set of $\Sigma$ can be, in some sense, localized;
for the definitions of $\bdd(\bs{e}) \cap \Sigma$, $\pi_b$ and $\gal_{\LL}^c(\Sigma, \bs{e})$, see preliminaries, Remark \ref{remark:bdd(e) cap Sigma}, Remark \ref{remark:diagram of maps}, and Definition \ref{definition: canonical sugbroups}.

\begin{maincor}
Let $\Sigma$ be an $\bs{e}$-invariant type and suppose $\bs{e} \in \dcl(\Sigma)$.
Let $c$ and $d$ be tuples of realizations of $\Sigma$.
\begin{enumerate}
    \item The following are equivalent.
    \begin{enumerate}
        \item $c \kpequiv_{\bs{e}} d$.
        \item $c \equiv_{\bdd(\bs{e}) \cap \Sigma} d$.
        \item There is $\sigma \in \pi_b^{-1}[\gal_{\LL}^c(\Sigma, \bs{e})]$ such that $\sigma(c) = d$.
    \end{enumerate}
    \item The following are equivalent.
    \begin{enumerate}
        \item $c \shequiv_{\bs{e}} d$
        \item $c \equiv_{\acl(\bs{e}) \cap \Sigma} d$.
        \item There is $\sigma \in \pi_b^{-1}[\gal_{\LL}^0(\Sigma, \bs{e})]$ such that $\sigma(c) = d$.
    \end{enumerate}
\end{enumerate}
\end{maincor}

Section \ref{sec:pre} is dedicated to presenting preliminaries.
In Subsection \ref{sec:pre hyperimaginary}, we recall definitions and facts about hyperimaginaries and Lascar groups.
In particular, the characterizations of Lascar, KP, and Shelah strong types over a hyperimagianry are given.
In the follwing Subsection \ref{sec:pre relativized}, we recall the definition of the relativized Galois groups of first-order theories over a hyperimaginary.

In Section \ref{sec:top on rel}, we develop most of the main results. We first introduce the notion of a Lascar tuple over a hyperimaginary $\bs{e}$ for an $\bs{e}$-invariant partial type $\Sigma$ and show its existence (Definition \ref{def:Lascar tuple} and Lemma \ref{lem:existence of Lascar tuple}). Under the assumption that there is a tuple $c$ of realizations of $\Sigma$ with $\bs{e}\in\dcl(c)$, we define a topology $\mathfrak{t}$ on a relativized Lascar group $\gall(\Sigma,\bs{e})$ of $\Sigma$, which is induced by the surjective map from the space of types over a Lascar tuple. The topology $\mathfrak{t}$ turns out to coincide with the quotient topology given from the ordinary Lascar group of the theory, and so $\gall(\Sigma,\bs{e})$ becomes a quasi-compact group with the topology $\mathfrak{t}$.

The key result in Section \ref{sec:top on rel} is the Galois theoretic correspondence between a closed subgroup of $\gall(\Sigma,\bs{e})$ (which can be characterized by its orbit equivalence relation) and a bounded hyperimaginary, whose set of representatives intersects with the set of tuples of realizations of $\Sigma$ (Proposition \ref{prop:description of closed subgroup} and Lemma \ref{lemma:finite characterization of orbit equiv}). The Lascar tuple plays an essential role in the proof of this correspondence, Proposition \ref{prop:description of closed subgroup}. As a corollary, we find a correspondence between the relativized KP [Shelah] strong type and the closure of the trivial subgroup [the connected component containing the identity] of the relativized Lascar group (Theorem \ref{thm:description of KP type} and Corollary \ref{cor:char of KP and SH in Sigma}).

In Section \ref{sec:rest and example}, we recall another method of considering automorphisms only in the solution set of a partial type by restricting the domain, and compare with the relativized one.
We prove that if the orbit equivalence relation induced from a given subgroup is weaker than having the same Lascar strong type, then the relativized subgroup of the automorphism group is the same as the smallest group containing the restricted one and the relativized Lascar group.

\section{Preliminaries}\label{sec:pre}

\subsection{Hyperimaginaries and model theoretic Galois groups}\label{sec:pre hyperimaginary}

Throughout, let $T$ be a complete first-order theory and $\mathfrak{C}$ be a monster model of $T$.
Recall that an object is said to be {\em small} if its size is less than the degree of strong homogeneity and saturation of $\mathfrak{C}$.
We will always consider small ordinals and cardinals, and small sets, tuples and models in $\mathfrak{C}$, unless otherwise stated.

Any reader who does not want to develop the arguments over a hyperimaginary (to avoid complexity) may assume that the hyperimaginary $\bs{e}$ throughout this paper to be an emptyset or interdefinable (Definition \ref{definition: interdefinable}) with an emptyset, but we need basic definitions and facts about hyperimaginaries to proceed the arguments even when $\bs{e} = \emptyset$.

The proofs for basic properties of hyperimaginaries can be found on \cite{C11} and \cite{K14}.
Most of the basic definitions and facts on the Lascar group can be found on \cite{KL23}, \cite{Lee22} and \cite{K14}, which collect and generalize the fundamental results in \cite{Z02}, \cite{CLPZ01}, \cite{LP01} and more papers, to the context of hyperimaginaries.

We denote the automorphism group of $\mathfrak{C}$ by $\aut(\mathfrak{C})$ and for $A\subseteq \mathfrak{C}$, we denote the set of automorphisms of $\mathfrak{C}$ fixing $A$ pointwise by $\aut_A(\mathfrak{C})$.

\begin{definition}\label{def:basic properties of equiv rel}
Let $E$ be an equivalence relation defined on $\mathfrak{C}^{\alpha}$, where $\alpha$ is an ordinal, so that its equivalence classes are sets of $\alpha$-tuples of elements of $\mathfrak{C}$.
Let $A \subseteq \mathfrak{C}$.
Then $E$ is said to be
\begin{enumerate}
    \item {\em finite} if the number of its equivalence classes is finite,
    \item {\em bounded} if the number of its equivalence classes is small,
    \item {\em $A$-invariant} if for any $f \in \aut_A(\mathfrak{C})$, $E(a, b)$ if and only if $E(f(a), f(b))$,
    \item {\em $A$-definable} if there is a formula $\varphi(x, y)$ over $A$ such that $\models \varphi(a, b)$ if and only if $E(a, b)$ holds, just {\em definable} if it is definable over some parameters, and
    \item {\em $A$-type-definable} if there is a partial type $\Phi(x, y)$ over $A$ such that $\models \Phi(a, b)$ if and only if $E(a, b)$ holds, just {\em type-definable} if it is type-definable over some parameters.
\end{enumerate}
\end{definition}

\begin{definition}\label{def:hyperimaginary}
Let $E$ be an $\emptyset$-type-definable equivalence relation on $\mathfrak{C}^{\alpha}$. An equivalence class of $E$ is called a {\em hyperimaginary} and it is denoted by $a_E$ if it has a representative $a$. A hyperimaginary $a_E$ is said to be {\em countable} if $|a|$ is countable.
\end{definition}

To make hyperimaginaries easily distinguishable from real tuples, we will write hyperimaginaries in boldface.

\begin{definition}\label{def:auto group} For a hyperimaginary $\bs{e}$,
$$\aut_{\bs{e}}(\mathfrak{C}):= \{f \in \aut(\mathfrak{C}): f(\bs{e}) = \bs{e} \text{ (setwise)}\}.$$

We say an equivalence relation $E$ is {\em $\bs{e}$-invariant} if for any $f \in \aute(\mathfrak{C})$, $E(a, b)$ holds if and only if $E(f(a), f(b))$ holds.
\end{definition}

\begin{definition}\label{definition: interdefinable}
For a hyperimaginary $\bs{e}$, we say that two `objects' (e.g. elements of $\mathfrak{C}$, tuples of equivalence classes, enumerations of sets) $b$ and $c$ are {\em interdefinable over $\bs{e}$} if for any $f \in \aute(\mathfrak{C})$, $f(b) = b$ if and only if $f(c) = c$.
We may omit `over $\bs{e}$' if $\bs{e} = \emptyset$.
\end{definition}


\begin{fact}[{\cite[Remark 4.1.2 and Lemma 4.1.3]{K14}} or {\cite[Lemmas 15.3 and 15.4]{C11}}]\label{fact:countable hyperimagianry is enough}\hspace{-1cm}$ $
\begin{enumerate}
    \item Any tuple (of elements) in $\mathfrak{C}$, any tuple of imaginaries of $\mathfrak{C}$, and any tuple of hyperimaginaries are interdefinable with a single hyperimaginary.
    \item Any hyperimaginary is interdefinable with a sequence of countable hyperimaginaries.
\end{enumerate}
\end{fact}

{\bf Until the end of this paper, we fix some arbitrary $\emptyset$-type-definable equivalence relation $E$ and a hyperimagianry $\bs{e} := a_E$.}

\begin{definition}\label{def:dcl, acl, bdd}$ $
\begin{enumerate}
    \item A hyperimaginary $\bs{e}'$ is {\em definable over $\bs{e}$} if $f(\bs{e}') = \bs{e}'$ for any $f \in \aute(\mathfrak{C})$.
    \item A hyperimaginary $\bs{e}'$ is {\em algebraic over $\bs{e}$} if $\{f(\bs{e}'): f \in \aute(\mathfrak{C})\}$ is finite.
    \item A hyperimaginary $\bs{e}'$ is {\em bounded over $\bs{e}$} if $\{f(\bs{e}'): f \in \aute(\mathfrak{C})\}$ is small.
    \item The {\em definable closure} of $\bs{e}$, denoted by $\dcl(\bs{e})$ is the set of all countable hyperimaginaries $\bs{e}'$ definable over $\bs{e}$.
    \item The {\em algebraic closure} of $\bs{e}$, denoted by $\acl(\bs{e})$ is the set of all countable hyperimaginaries $\bs{e}'$ algebraic over $\bs{e}$.
    \item The {\em bounded closure} of $\bs{e}$, denoted by $\bdd(\bs{e})$ is the set of all countable hyperimaginaries $\bs{e}'$ bounded over $\bs{e}$.
\end{enumerate}
\end{definition}

\begin{def/rem}\label{def/rem: dcl acl bdd}$ $
\begin{enumerate}
    \item By Fact \ref{fact:countable hyperimagianry is enough}(2), if $f \in \aut(\mathfrak{C})$ fixes $\bdd(\bs{e})$ elementwise, then for any hyperimaginary $\bs{e}'$ which is bounded over $\bs{e}$, $f(\bs{e}') = \bs{e}'$.
    Similar statements also hold for $\dcl(\bs{e})$ and $\acl(\bs{e})$.
    \item By (1), for a hyperimaginary $b_F$ which is possibly not countable, we write $b_F \in \bdd(\bs{e})$ if $b_F$ is bounded over $\bs{e}$.
    We use notation $b_F \in \dcl(\bs{e})$, $\acl(\bs{e})$ in a similar way.
    \item Each of $\dcl(\bs{e})$, $\acl(\bs{e})$, and $\bdd(\bs{e})$ is small and interdefinable with a single hyperimaginary (cf. \cite[Proposition 15.18]{C11}).
    Thus it makes sense to consider $\aut_{\bdd(\bs{e})}(\mathfrak{C})$.
\end{enumerate}
\end{def/rem}

The complete type of a hyperimaginary over a hyperimaginary, and the equality of such complete types are type-definable:

\begin{definition}[{\cite[Section 4.1]{K14}}]\label{definition: hyperimaginary type}
Let $b_F$ and $c_F$ be hyperimaginaries.
\begin{enumerate}
    \item The {\em complete type of $b_F$ over $\bs{e}$}, $\tp_x(b_F/\bs{e})$ is a partial type over $a$,
    $$\exists z_1 z_2 (\tp_{z_1 z_2}(ba) \wedge F(x, z_1) \wedge E(a, z_2)),$$    
     whose solution set is the union of automorphic images of $b_F$ over $\bs{e}$.
    \item We write $b_F\equiv_{\bs{e}} c_F$ if there is an automorphism $f\in \aute(\mathfrak{C})$ such that $f(b_F)=c_F$. 
\end{enumerate}
\end{definition}

\begin{remark}\label{rem:type-definabilityt_equility of types}
The equivalence relation $x_F\equiv_{\bs{e}} y_F$ in variables $xy$ is $a$-type-definable, given by the partial type:
    $$\exists z_1z_2w_1w_2(E(a, z_1) \wedge E(a, z_2) \wedge \tp(z_1w_1) = \tp(z_2w_2) \wedge F(w_1, x) \wedge F(w_2, y)).$$ By abusing notation, we also sometimes use notation $\tp_x(b_F/{\bs{e}})=\tp_x(c_F/{\bs{e}})$ instead of $b_F\equiv_{\bs{e}} c_F$.
\end{remark}

Note that for hyperimaginaries $\bs{a}$ and $\bs{b}$ with $\bs{a}\in \dcl(\bs{b})$, $$\aut_{\bs{b}}(\mathfrak{C})\le\aut_{\bs{a}}(\mathfrak{C}).$$ Now we start to recall the model theoretic Galois groups.

\begin{definition}\label{def:Lascar group}$ $
\begin{enumerate}
    \item $\autfl(\mathfrak{C}, \bs{e})$ is a normal subgroup of $\aute(\mathfrak{C})$ generated by
    \begin{center}
    $\{f \in \aute(\mathfrak{C}): f \in \aut_M(\mathfrak{C})$ for some $M \models T \text{ such that } \bs{e} \in \dcl(M)\},$
    \end{center}
    and its elements are called {\em Lascar strong automorphims over $\bs{e}$}.
    \item The quotient group $\gall(T, \bs{e}) = \aute(\mathfrak{C})/\autfl(\mathfrak{C}, \bs{e})$ is called the {\em Lascar group} of $T$ over $\bs{e}$.
\end{enumerate}
\end{definition}

\begin{fact}[{\cite[Section 1]{KL23}}]\label{fact:Lascar group is small}$ $
\begin{enumerate}
    \item $\gall(T, \bs{e})$ does not depend on the choice of a monster model up to isomorphism, so it is legitimate to write $\gall(T, \bs{e})$ instead of $\gall(\mathfrak{C}, \bs{e})$.
    \item $[\aute(\mathfrak{C}):\autfl(\mathfrak{C}, \bs{e})] = |\gall(T, \bs{e})| \le 2^{|T| + |a|}$, which is small.
\end{enumerate}
\end{fact}

\begin{fact}[{\cite[Lemma 1]{Z02}}]\label{fact:nu is well-defined}
Let $M$ be a small model of $T$ such that $\bs{e} \in \dcl(M)$ and $f, g \in \aute(\mathfrak{C})$.
If $\tp(f(M)/M) = \tp(g(M)/M)$, then $f \cdot \autfl(\mathfrak{C}, \bs{e}) = g \cdot \autfl(\mathfrak{C}, \bs{e})$ as elements in $\gall(T, \bs{e})$.
\end{fact}

\begin{definition}
Let $M$ be a small model of $T$ such that $\bs{e} \in \dcl(M)$ so that for any tuples $a$ and $b$ in $\mathfrak{C}$, if $a\equiv_M b$, then $a\equiv_{\bs{e}}b$.
\begin{enumerate}
    \item $S_M(M) = \{\tp(f(M)/M): f \in \aute(\mathfrak{C})\}$.
    \item $\nu: S_M(M) \rightarrow \gall(T, \bs{e})$ is defined by $$\nu(\tp(f(M)/M)) := f \cdot \autfl(\mathfrak{C}, \bs{e}),$$ which is well-defined by Fact \ref{fact:nu is well-defined}, and we write $[f]:=f\cdot \autfl(\mathfrak{C}, \bs{e})$.
    \item $\mu: \aute(\mathfrak{C}) \rightarrow S_M(M)$ is defined by $\mu(f) = \tp(f(M)/M)$.
    \item $\pi = \nu \circ \mu: \aute(\mathfrak{C}) \rightarrow \gall(T, \bs{e})$, so that $\pi(f) = [f]$ in $\gall(T, \bs{e})$.
\end{enumerate}
\end{definition}

\begin{remark}[{\cite[Remark 1.9]{KL23}}]\label{remark:type space compact}
Let $S_x(M) = \{p(x): |x| = |M|$ and $ p(x)$ is a complete type over $M\}$ be the compact space of complete types.
Note that even if $\bs{e} \in \dcl(M)$, possibly $a$ is not in $M$, so that $S_M(M)$ is {\em not} $\{p \in S_x(M): \tp(M/\bs{e}) \subseteq p\}$.
But for any small model $M$ such that $\bs{e} \in \dcl(M)$, $S_M(M)$ is a closed (so compact) subspace of $S_x(M)$.
\end{remark}
\begin{proof}
Since $\tp(M/\bs{e})$ is $\bs{e}$-invariant and $\bs{e} \in \dcl(M)$, there is $\Gamma(x, M)$ over $M$, whose solution set is the same as $\tp(M/\bs{e})$.
Then
$$S_M(M) = \{\tp(f(M)/M): f \in \aute(\mathfrak{C})\}= \{p(x) \in S_x(M): \Gamma(x, M) \subseteq p(x)\},$$
which is closed.

\end{proof}

\begin{fact}[{\cite[Proposition 1.10 and Corollary 1.20]{KL23}}]\label{fact:indep of M and top gp}$ $
\begin{enumerate}
    \item We give the quotient topology on $\gall(T, \bs{e})$ induced by $\nu: S_M(M) \rightarrow \gall(T, \bs{e})$.
    The quotient topology does not depend on the choice of $M$.
    \item $\gall(T, \bs{e})$ is a quasi-compact topological group.
\end{enumerate}
\end{fact}

Now we introduce the notation for two canonical normal subgroups of the Lascar group and recall the KP and Shelah group.

\begin{definition}
Let $\gall^c(T, \bs{e})$ be the topological closure of the trivial subgroup of $\gall(T, \bs{e})$, and $\gall^0(T, \bs{e})$ be the connected component containing the identity in $\gall(T, \bs{e})$. Put $\autfkp(\mathfrak{C}, \bs{e}) = \pi^{-1}[\gall^c(T, \bs{e})]$ and $\autfsh(\mathfrak{C}, \bs{e}) = \pi^{-1}[\gall^0(T, \bs{e})]$.
\begin{enumerate}
    \item The {\em KP(-Galois) group of $T$} over $\bs{e}$ (where KP is an abbreviation for Kim-Pillay) is
    $$\galkp(T, \bs{e}) := \aute(\mathfrak{C}) / \autfkp(\mathfrak{C}, \bs{e}).$$
    \item The {\em Shelah(-Galois) group of $T$} over $\bs{e}$ is $$\galsh(T, \bs{e}) := \aute(\mathfrak{C}) / \autfsh(\mathfrak{C}, \bs{e}).$$
\end{enumerate}
\end{definition}
\noindent Note that we have
$$\galkp(T, \bs{e})\cong \gall(T, \bs{e}) / \gal_{\LL}^c(T, \bs{e}),\ \galsh(T, \bs{e})\cong \gall(T, \bs{e}) / \gal_{\LL}^0(T, \bs{e}).$$

The strong types are defined in terms of equivalence relations satisfying particular conditions, but because we are considering a Lascar group over a hyperimaginary $\bs{e}$, it is beneficial to define them as the orbit equivalence relations and then characterize them.

\begin{definition}
Given hyperimaginaries $b_F, c_F$, they are said to have the same
\begin{enumerate}
    \item {\em Lascar strong type} over $\bs{e}$ if there is $f \in \autfl(\mathfrak{C}, \bs{e})$ such that $f(b_F) = c_F$, and it is denoted by $b_F \lequiv_{\bs{e}} c_F$,
    \item {\em KP strong type} over $\bs{e}$ if there is $f \in \autfkp(\mathfrak{C}, \bs{e})$ such that $f(b_F) = c_F$, and it is denoted by $b_F \kpequiv_{\bs{e}} c_F$, and
    \item {\em Shelah strong type} over $\bs{e}$ if there is $f \in \autfsh(\mathfrak{C}, \bs{e})$ such that $f(b_F) = c_F$, and it is denoted by $b_F \sequiv_{\bs{e}} c_F$.
\end{enumerate}
\end{definition}

\begin{fact}[{\cite[Remark 1.4(2), Proposition 3.5, 3.6, and 3.12]{KL23}}]\label{fact:characterization of equivalence relations}
Let $b_F$ and $c_F$ be hyperimaginaries.
\begin{enumerate}
    \item $b_F \lequiv_{\bs{e}} c_F$ if and only if for any $\bs{e}$-invariant bounded equivalence relation $E'$ which is coarser than $F$, $E'(b, c)$ holds; $x_F \lequiv_{\bs{e}} y_F$ is the finest such equivalence relation.
    \item The following are equivalent.
    \begin{enumerate}
        \item $b_F \kpequiv_{\bs{e}} c_F$.
        \item $b_F \equiv_{\bdd(\bs{e})} c_F$.
        \item For any $\bs{e}$-invariant type-definable bounded equivalence relation $E'$ which is coarser $F$, $E'(b, c)$ holds ($x_F \kpequiv_{\bs{e}} y_F$ is the finest such equivalence relation).
    \end{enumerate}
    \item The following are equivalent.
    \begin{enumerate}
        \item $b_F \sequiv_{\bs{e}} c_F$.
        \item $b_F \equiv_{\acl(\bs{e})} c_F$.
        \item For any $\bs{e}$-invariant type-definable equivalence relation $E'$ coarser than $F$, if $b_{E'}$ has finitely many conjugates over $\bs{e}$, then $E'(b, c)$ holds.
    \end{enumerate}
\end{enumerate}
If $F$ is just =, so that $b_F$ and $c_F$ are just real tuples, then we can omit ``coarser than $F$''.
\end{fact}

The following observation is necessary when we consider an equivalence class of a type-definable equivalence relation over certain parameters as a hyperimaginary.

\begin{Fact}\label{fact:e-hi is empty hi}
Let $F$ be an $\bs{e}$-invariant type-definable equivalence relation on $\mathfrak{C}^{\alpha}$ and assume that $\bs{e} \in \dcl(b)$ for some tuple $b$ of elements in $\mathfrak{C}$ (note that we always have at least one such $b$ by letting $b = a$).
Then there is an $\emptyset$-type-definable equivalence relation $F'$ such that for any $c \in \mathfrak{C}^\alpha$, $c_F$ and $(cb)_{F'}$ are interdefinable over $\bs e$ so that we can `replace' an equivalence class of an $\bs{e}$-invariant type-definable equivalence relation with a hyperimaginary.
\end{Fact}
\begin{proof}
Since $\bs{e} \in \dcl(b)$, $\bs{e} = E(\mathfrak{C}, a)$ is $b$-invariant, so there is a $b$-type-definable equivalence relation $E^*$ such that $b/E^* = E^*(\mathfrak{C}, b) = E(\mathfrak{C}, a) = \bs{e}$.
Furthermore, because $F$ is $\bs{e}$-invariant and $\bs{e} \in \dcl(b)$, $F$ is type-definable over $b$, say by $F(x,y;b)$.
Then for $p(x) = \tp(b)$, put
$$F'(xz,yw) := (F(x,y;z) \wedge E^*(z, w) \wedge p(z) \wedge p(w)) \vee xz=yw.$$
It is not difficult to verify that  $F'$ is indeed an equivalence relation; we claim that $F'$ is the desired one.
For any $f \in \aute(\mathfrak{C})$, $\models F'(cb, f(cb))$ if and only if
$$\models (F(c, f(c); b) \wedge E^*(b, f(b)) \wedge p(b) \wedge p(f(b))) \vee cb = f(cb),$$
which is just equivalent to $\models F(c, f(c);b) \vee cb = f(cb)$.
Now it is immediate that $f \in \aut_{c_F \bs{e}}(\mathfrak{C})$ if and only if $f \in \aut_{(cb)_{F'} \bs{e}}(\mathfrak{C})$.
\end{proof}

\subsection{Relativized model theoretic Galois groups}\label{sec:pre relativized}
From now on, we fix an $\bs{e}$-invariant partial type $\Sigma(x)$ where $x$ is a possibly infinite tuple of variables.
That is, for any $f \in \aute(\mathfrak{C})$, $b \models \Sigma(x)$ if and only if $f(b) \models \Sigma(x)$. We write $\Sigma(\mathfrak{C})$ for the set of tuples $b$ of elements in $\mathfrak{C}$ such that $b \models \Sigma(x)$.

\begin{definition}[Relativization of automorphism groups to $\Sigma$]\label{def:relativized Galois groups}
Let $X \in \{\LL, \KP, \SH\}$.
\begin{enumerate}
    \item  $\aute(\Sigma) = \aute(\Sigma(\mathfrak{C})) = \{f \upharpoonright \Sigma(\mathfrak{C}): f \in \aute(\mathfrak{C})\}$.
    \item For a cardinal $\lambda$, $\autf^{\lambda}_X(\Sigma, \bs{e}) = $
    $$\{\sigma \in \aute(\Sigma): \text{ for any of tuple } b = (b_i)_{i<\lambda}\text{ where each }b_i \models \Sigma(x_i)\text{, }b \equiv^X_{\bs{e}} \sigma(b)\}.$$
    \item $\autf_X(\Sigma, \bs{e}) =$
    \begin{align*}
    \{\sigma \in \aute(\Sigma): & \text{ for any cardinal } \lambda \text{ and for any tuple } b = (b_i)_{i<\lambda}\\
    & \text{ with each }b_i \models \Sigma(x_i)\text{, }b \equiv^X_{\bs{e}} \sigma(b)\}.
    \end{align*}
\end{enumerate}
\end{definition}

%

\begin{remark}
For $X \in \{\LL, \KP, \SH\}$, it is easy to check that $\autf^{\lambda}_X(\Sigma, \bs{e})$ and $\autf_X(\Sigma, \bs{e})$ are normal subgroups of $\aute(\Sigma)$.
\end{remark}

\begin{definition}$ $
\begin{enumerate}
    \item For $X \in \{\LL, \KP, \SH\}$, for any cardinal $\lambda$, $\gal_X^{\lambda}(\Sigma, \bs{e}) = \aute(\Sigma)/\autf_X^{\lambda}(\Sigma, \bs{e})$.
    \item $\gall(\Sigma, \bs{e}) = \aute(\Sigma)/\autfl(\Sigma, \bs{e})$ is the {\em Lascar(-Galois) group over $\bs{e}$ relativized to $\Sigma$}.
    \item $\galkp(\Sigma, \bs{e}) = \aute(\Sigma)/\autfkp(\Sigma, \bs{e})$ is the {\em KP(-Galois) group over $\bs{e}$ relativized to $\Sigma$}.
    \item $\galsh(\Sigma, \bs{e}) = \aute(\Sigma)/\autfsh(\Sigma, \bs{e})$ is the {\em Shelah(-Galois) group over $\bs{e}$ relativized to $\Sigma$}.
\end{enumerate}
\end{definition}

\begin{remark}\label{remark:original Galois group is finer}$ $
\begin{enumerate}
    \item For $X \in \{\LL, \KP, \SH\}$, if $[f] = [\id]$ in $\gal_X(T, \bs{e})$, then $[f] = [\id]$ in $\gal_X(\Sigma, \bs{e})$.
    \item In general, $\gal_{\LL}^1(\Sigma) \ne \gal_{\LL}^2(\Sigma)$ (cf. \cite[Example 2.3]{DKKL21}).
\end{enumerate}
\end{remark}


\begin{fact}\label{fact:small tuple enough}$ $
\begin{enumerate}
    \item $\autfl(\Sigma, \bs{e})=\autf_{\LL}^{\omega}(\Sigma, \bs{e})$.
    \item $\autfkp(\Sigma, \bs{e})=\autf_{\KP}^{\omega}(\Sigma, \bs{e})$.
    \item $\autfsh(\Sigma, \bs{e})=\autf_{\SH}^{\omega}(\Sigma, \bs{e})$.
\end{enumerate}
\end{fact}

\begin{proof}
(1): The proof is the same as \cite[Remark 3.3]{DKL17} and \cite[Proposition 6.3]{Lee22}, but we prove it again here because their statements are (seemingly) slightly weaker and the proof requires some facts not contained in this paper.
Let $\alpha \ge \omega$ be a small ordinal and $b = (b_i)_{i<\alpha}$, $c = (c_i)_{i<\alpha}$ with each $b_i \models \Sigma(x_i)$, $c_i \models \Sigma(x_i)$.
We will prove that if their corresponding countable subtuples have the same Lascar strong type over $\bs{e}$, then $b \lequiv_{\bs{e}} c$, which will complete the proof of (1).

We will use the following facts over a hyperimaginary (\cite[Definition 4.4 and Remark 4.5(1)(4)]{KL23} with $F$ being ``$=$''):
\begin{itemize}
    \item For any tuples $b$ and $c$ of $\mathfrak{C}$, $b \lequiv_{\bs{e}} c$ if and only if $d_{\bs{e}}(b, c) \le n$ for some $n < \omega$, where $d_{\bs{e}}(b, c)$ is called the {\em Lascar distance} between $b$ and $c$, and $d_{\bs{e}}(b, c) \le n$ can be expressed as a type over $abc$.
    Also, $d_{\bs{e}}(b, c) \le n$ if and only if $d_{\bs{e}}(b', c') \le n$ for any corresponding subtuples $b'$ and $c'$ of $b$ and $c$.
\end{itemize}

Assume that the assertion holds for all infinite ordinals less than $\alpha$, and the corresponding countable subtuples of $b =  (b_i)_{i<\alpha}$ and $c = (c_i)_{i<\alpha}$ have the same Lascar strong type over $\bs{e}$, so that for each $\beta < \alpha$, there is $n_{\beta} < \omega$ such that $d_{\bs{e}}((b_i)_{i<\beta}, (c_i)_{i<\beta}) \le n_{\beta}$ by the inductive hypothesis.

Suppose that there is a sequence of ordinals $(\beta_k)_{k < \omega}$ such that $\beta_k < \alpha$ and $n_{\beta_k} \ge k$ for each $k < \omega$, for a contradiction.
For any subset $I$ of $\alpha$, let $b_{I}$ denote the subsequence $(b_i: i \in I)$.
If it is true that for any finite subset $I_k$ of $\beta_k$, $d_{\bs{e}}(b_I, c_I) \le k - 1$, then by compactness, $d_{\bs{e}}(b_{\beta_k}, c_{\beta_k}) \le k - 1$, which is a contradiction.
Thus, for each $\beta_k$, there is a finite subset $I_k$ such that $d_{\bs{e}}(b_{I_k}, c_{I_k}) \ge k$.
Let $I = \bigcup_{k < \omega} I_k$, which is a countable subset of $\alpha$.
Then $d_{\bs{e}}(b_I, c_I) \ge k$ for any $k < \omega$, which is impossible by the assumption.
Hence there is $k_0 < \omega$ such that for any $\beta < \alpha$, $n_{\beta} \le k_0$, thus by compactness, $d_{\bs{e}}(b, c) \le k_0$.

(2): $\sigma \in \autfkp(\Sigma, \bs{e})$ if and only if for any tuple $b$ of realizations of $\Sigma$, $b \kpequiv_{\bs{e}} \sigma(b)$.
But $\kpequiv_{\bs{e}}$ is equivalent to $\equiv_{\bdd(\bs{e})}$ (which is type-definable) by Fact \ref{fact:characterization of equivalence relations}, thus if $b' \kpequiv_{\bs{e}} \sigma(b')$ for any corresponding subtuples $b', \sigma(b')$ of $b, \sigma(b)$, which are tuples of finitely many realizations of $\Sigma$, then $b \kpequiv_{\bs{e}} \sigma(b)$ by compactness.
Thus if $\sigma \in \autf_{\KP}^{\omega}(\Sigma, \bs{e})$, then $\sigma \in \autfkp(\Sigma, \bs{e})$.
The proof for (3) is the same as (2), replacing $\bdd(\bs{e})$ by $\acl(\bs{e})$.
\end{proof}

\begin{fact}\label{fact:indep of monster model}
The relativized Lascar group $\gall(\Sigma, \bs{e})$, relativized KP group $\galkp(\Sigma, \bs{e})$, and relativized Shelah group $\galsh(\Sigma, \bs{e})$ do not depend on the choice of $\mathfrak{C}$.
\end{fact}
\begin{proof}
Basically, the proof of \cite[Proposition 3.6]{DKL17} or \cite[Proposition 6.5]{Lee22} can be used, but the situations are not exactly the same as here.
Here we prove for $\gal_{\KP}(\Sigma, \bs{e})$ in the same way with full details.
The proof for $\gall(\Sigma, \bs{e})$ and $\gal_{\SH}(\Sigma, \bs{e})$ can be obtained by replacing KP with L and S in the following proof, respectively.

Let $\mathfrak{C} \prec \mathfrak{C}'$ be monster models.
Define $\eta: \galkp(\Sigma(\mathfrak{C}), \bs{e}) \rightarrow \galkp(\Sigma(\mathfrak{C}'), \bs{e})$ by
$$f \upharpoonright \Sigma(\mathfrak{C}) \cdot \autfkp(\Sigma(\mathfrak{C}), \bs{e}) \mapsto f' \upharpoonright \Sigma(\mathfrak{C}') \cdot \autfkp(\Sigma(\mathfrak{C}'), \bs{e}),$$
$$\text{or, in short, }[f] \mapsto [f']$$
where $f' \in \aute(\mathfrak{C}')$ is any extension of $f \in \aute(\mathfrak{C})$.

We first check that it is well-defined.
It suffices to show the following:
For any $f \in \aute(\mathfrak{C})$, if $f \upharpoonright \Sigma(\mathfrak{C}) \in \autfkp(\Sigma(\mathfrak{C}), \bs{e})$, then $f' \upharpoonright \Sigma(\mathfrak{C}') \in \autfkp(\Sigma(\mathfrak{C}'), \bs{e})$.
First, consider $f \in \aute(\mathfrak{C})$ such that $f \upharpoonright \Sigma(\mathfrak{C}) \in \autfkp(\Sigma(\mathfrak{C}), \bs{e})$ and any extension $f'$ to $\mathfrak{C}'$.
Let $(b'_i)_{i < \lambda}$ be a tuple of realizations of $\Sigma$ in $\mathfrak{C}'$, and $(b_i)_{i < \lambda}$ in $\mathfrak{C}$ realize $\tp((b'_i)_{i < \lambda}/M)$ where $M$ is a small model containing $a$ (so that $\bs{e} \in \dcl(M))$.
Then $(b'_i)_{i < \lambda} \lequiv_{\bs{e}} (b_i)_{i < \lambda}$, and thus
$$f'((b'_i)_{i < \lambda}) \lequiv_{\bs{e}} f'((b_i)_{i < \lambda}) = f((b_i)_{i < \lambda}) \kpequiv_{\bs{e}} (b_i)_{i < \lambda} \lequiv_{\bs{e}} (b'_i)_{i < \lambda}.$$
Thus $f' \upharpoonright \Sigma(\mathfrak{C}') \in \autfkp(\Sigma(\mathfrak{C}'), \bs{e})$.

It is straightforward to see that $\eta$ is a homomorphism:
For any $f, g \in \aute(\mathfrak{C})$ and extensions $f', g' \in \aute(\mathfrak{C}')$, $f' \circ g'$ is an extension of $f \circ g$, hence $\eta([f \circ g]) = [f' \circ g']$ in $\galkp(\mathfrak{C}', \bs{e})$ by well-definedness.

To show that it is injective, let $f_1, f_2 \in \aute(\mathfrak{C})$, $f'
_1, f'_2 \in \aute(\mathfrak{C}')$ be any extensions of them and assume $[f'_1] = [f'_2]$ in $\galkp(\Sigma(\mathfrak{C}'), \bs{e})$.
Then $f_1'f_2'^{-1} \in \autfkp(\mathfrak{C}', \bs{e})$, so for any tuple $(b_i)_{i < \lambda}$ of realizations of $\Sigma$ in $\mathfrak{C}$, $f_1f_2^{-1}((b_i)_{i < \lambda}) = f_1'f_2'^{-1}((b_i)_{i < \lambda}) \kpequiv_{\bs{e}} (b_i)_{i < \lambda}$.

For surjectivity, consider any $f^* \in \aute(\mathfrak{C}')$ and small model $M$ containing $a$.
Choose any $f \in \aute(\mathfrak{C})$ such that $\tp(f(M)/M) = \tp(f^*(M)/M)$.
Then by Fact \ref{fact:nu is well-defined}, for any extension $f'$ of $f$, $\overline{f'} = \overline{f^*}$ in $\galkp(\Sigma(\mathfrak{C}'), \bs{e})$.
\end{proof}

\section{Topology on relativized model theoretic Galois groups}\label{sec:top on rel}
In this section, we will generalize Fact \ref{fact:indep of M and top gp} and Fact \ref{fact:characterization of equivalence relations} to the relativized Lascar group. To that end, we will find a topology $\mathfrak{t}$ on the relativized Lascar group $\gall(\Sigma,\bs{e})$, so that
\begin{itemize}
    \item the group $\gall(\Sigma,\bs{e})$ is a quasi-compact group with respect to $\mathfrak{t}$, and
    \item $\autfkp(\Sigma, \bs{e})=\pi'^{-1}[\gall^c(\Sigma,\bs{e})]$ and $\autfsh(\Sigma, \bs{e})=\pi'^{-1}[\gall^0(\Sigma,\bs{e})]$,
\end{itemize}
where $\pi':\aute(\Sigma)\rightarrow \gall(\Sigma,\bs{e})$ is the natural surjective map, and $\gall^c(\Sigma,\bs{e})$ and $\gall^0(\Sigma,\bs{e})$ are the closure of the trivial subgroup and the connected component containing the identity of $\gall(\Sigma,\bs{e})$ with respect to the topology $\mathfrak{t}$, respectively.

For this purpose, we introduce a notion that will replace the role of a small model when we consider Lascar groups.

\begin{definition}\label{def:Lascar tuple}
A small tuple $b$ of realizations of $\Sigma$ is called a {\em Lascar tuple} (in $\Sigma$ over $\bs{e}$) if $\autfl(\Sigma, \bs{e})=\{\sigma \in \aute(\Sigma) :b \lequiv_{\bs{e}} \sigma(b)\}$.
\end{definition}

\begin{lemma}\label{lem:existence of Lascar tuple}
For any $\bs{e}$-invariant partial type $\Sigma(x)$, there is a Lascar tuple.
\end{lemma}

\begin{proof}
By Fact \ref{fact:Lascar group is small}(2), the set of all Lascar strong types
$$C = \{c_{\lequiv_{\bs{e}}}: c \text{ is a countable tuple of realizations of }\Sigma\}$$
is small, say its cardinality is $\kappa$.

Let $b = (b_i: i < \kappa)$ be a small tuple that collects representatives of Lascar strong types in $C$, only one for each equivalence class.
Then $b$ is the desired one; by Fact \ref{fact:small tuple enough}, it is enough to show that for an automorphism $f \in \aute(\mathfrak{C})$, if $b \lequiv_{\bs{e}} f(b)$, then for any countable tuple $d$ of realizations of $\Sigma$, $d \lequiv_{\bs{e}} f(d)$.
Note that there is $i < \kappa$ such that $d \lequiv_{\bs{e}} b_i$.
Then we have
$$d \lequiv_{\bs{e}} b_i \lequiv_{\bs{e}} f(b_i) \lequiv_{\bs{e}} f(d)$$
where the last equivalence follows from the invariance of $\lequiv_{\bs{e}}$.
\end{proof}

\begin{remark}\label{rem:basic properties on Lascar tuples}$ $
\begin{enumerate}
    \item Any small tuple of realizations of $\Sigma$ can be extended into a Lascar tuple.  
    \item Any small tuple of realizations of $\Sigma$ containing a Lascar tuple is a Lascar tuple.
\end{enumerate} 
\end{remark}

Fix a Lascar tuple $b$ in $\Sigma$ over $\bs{e}$, a small model $M$ with $b\in M$. Let 
\begin{align*}
S_b(b) & := \{\tp(\sigma(b)/b): \sigma \in \aute(\Sigma)\} \\
& = \{\tp(f(b)/b): f \in \aute(\mathfrak{C})\}.
\end{align*}
By the natural restriction map $r:S_M(M)\rightarrow S_b(b)$, which is continuous with respect to the logic topology on the type space, we have that $S_b(b)$ is a compact space.

\begin{remark}\label{remark:e in dcl(b) is necessary}
Assume that $\bs{e} \in \dcl(b)$ for some Lascar tuple $b$ in $\Sigma$ over $\bs{e}$.
Then the map
$$\nu_b:S_b(b)\rightarrow \gal_{\LL}^{\lambda}(\Sigma, \bs{e}),\ p = \tp(\sigma(b)/b) \mapsto [\sigma].$$
is well-defined.
\end{remark}

\begin{proof}
Given $\sigma,\sigma'\in \aute(\Sigma)$ with $\sigma(b)\equiv_{b}\sigma'(b)$, there is $\tau \in \aut_b(\Sigma)$ such that $\tau(\sigma'(b))=\sigma(b)$, and so $\sigma^{-1}\tau \sigma'\in \aut_b(\Sigma)$. Since $\bs{e}\in \dcl(b)$ and $b$ is a Lascar tuple, we have $$\aut_b(\Sigma)=\aut_{b\bs{e}}(\Sigma)\le\autfl(\Sigma,\bs{e}),$$
and
$$[\id]=[\sigma^{-1}\tau \sigma']=[\sigma]^{-1}[\tau][\sigma']=[\sigma]^{-1}[\sigma'] \mbox{, hence } [\sigma]=[\sigma'].$$
\end{proof}

{\bf From now on, we assume that there is a Lascar tuple $b$ in $\Sigma$ over $\bs{e}$ such that $\bs{e} \in \dcl(b)$.} Note that it is equivalent to the condition that $\bs{e}\in\dcl(\Sigma)$, that is, there is a small tuple $c$ of realizations of $\Sigma$ such that $\bs{e}\in\dcl(c)$; if there is such a small tuple $c$, by Remark \ref{rem:basic properties on Lascar tuples}, we may assume that $c$ is a Lascar tuple.
This assumption is explicitly used in Remark \ref{remark:e in dcl(b) is necessary}, and in the proofs of Lemma \ref{lemma:kp and sh are closed} and Proposition \ref{prop:id clousre, component finest}.

\begin{remark}\label{remark:diagram of maps}
We have the following commutative diagram of natural surjective maps:
\begin{center}
\begin{tikzcd}
{\aute(\mathfrak{C})} \arrow[r, "\mu"] \arrow[d, "\xi"] & {S_M(M)} \arrow[r, "\nu"] \arrow[d, "r"] & {\gall(T, \bs{e})} \arrow[r, "\eta_{\KP}"] \arrow[d, "\xi_{\LL}"] & {\galkp(T, \bs{e})} \arrow[r, "\eta_{\SH}"] \arrow[d, "\xi_{\KP}"] & {\galsh(T, \bs{e})} \arrow[d, "\xi_{\SH}"]\\
{\aute(\Sigma)} \arrow[r, "\mu_b"] & {S_b(b)} \arrow[r, "\nu_b"] & {\gal_{\LL}(\Sigma, \bs{e})} \arrow[r, "\eta_{\KP, \Sigma}"] & {\gal_{\KP}(\Sigma, \bs{e})} \arrow[r, "\eta_{\SH, \Sigma}"] & {\gal_{\SH}(\Sigma, \bs{e})}
\end{tikzcd}
\end{center}
Put $\pi_b:= \nu_b \circ \mu_b$ and $\pi_{\Sigma}:= \pi_b \circ \xi = \xi_{\LL} \circ \pi$.
Note that the projection maps $\pi=\nu\circ \mu$ and $\pi'=\pi_b$, and the following are group homomorphisms:
\begin{itemize}
	\item $\xi:\aute(\mathfrak{C})\rightarrow \aute(\Sigma)$,
	\item $\xi_X:\gal_X(T,\bs{e})\rightarrow \gal_X(\Sigma,\bs{e})$ for $X\in\{\mbox{L}, \mbox{KP}, \mbox{S}\}$,
	\item $\pi:\aute(\mathfrak{C})\rightarrow \gall(T,\bs{e})$,
	\item $\pi_{\Sigma}:\aute(\mathfrak{C})\rightarrow \gall(\Sigma,\bs{e})$, and
	\item $\pi':\aute(\Sigma)\rightarrow \gall(\Sigma,\bs{e})$.
\end{itemize}
Also, note that possibly $\xi[\autfl(T,\bs{e})]\subsetneq \autfl(\Sigma,\bs{e})$, which will be discussed in Section \ref{sec:rest and example}.

\end{remark}

\begin{remark}[Relativized Galois groups are topological groups]\label{relativized Lascar is top gp}
Consider a topology $\mathfrak{t}_b$ on $\gal_{\LL}(\Sigma, \bs{e})$ given by the quotient topology via $\nu_b$. Then $(\gal_{\LL}(\Sigma, \bs{e}),\mathfrak{t}_b)$ is a quasi-compact topological group whose topology $\mathfrak{t}_b$ is independent of the choice of a Lascar tuple $b$, so there is no harm to denote this $\mathfrak{t}_b$ by $\mathfrak{t}$.
\end{remark}

\begin{proof}
The restriction map $r: S_M(M) \rightarrow S_b(b)$ is a continuous surjective map between compact Hausdorff spaces, hence a quotient map.
Thus $\nu_b: S_b(b) \rightarrow \gal_{\LL}(\Sigma, \bs{e})$ and $\nu_b \circ r:S_M(M) \rightarrow \gal_{\LL}(\Sigma, \bs{e})$ induce the same quotient topology on $\gal_{\LL}(\Sigma, \bs{e})$.

But the quotient topology on $\gal_{\LL}(\Sigma, \bs{e})$ given by the natural projection map $\xi_{\LL}: \gall(T, \bs{e}) \rightarrow \gal_{\LL}(\Sigma, \bs{e})$ is also the same as the above topology, thus the topology of $\gal_{\LL}(\Sigma, \bs{e})$ is independent of the choice $b$ (Fact \ref{fact:indep of M and top gp}) and it is a topological group since a quotient group of a topological group with quotient topology is a topological group (it is the same reasoning as \cite[Remark 3.4]{DKL17}).
\end{proof}


The following proposition is a relativized analogue of \cite[Proposition 2.2]{KL23}, which have extended \cite[Lemma 1.9]{LP01} over a hyperimaginary.
The purpose of Proposition \ref{prop:description of closed subgroup} is to interpret closed subgroups of $\gall(\Sigma, \bs{e})$ using bounded hyperimaginaries, and it will play a critical role in many arguments until the end of this section.

\begin{prop}\label{prop:description of closed subgroup}
Let $H\leq \aute(\mathfrak{C})$.
The following are equivalent.
\begin{enumerate}
    \item $\pi_{\Sigma}[H]$ is closed in $\gall(\Sigma, \bs{e})$ and $H = \pi_{\Sigma}^{-1}[\pi_{\Sigma}[H]]$.
    \item $H = \aut_{\bs{e}'\bs{e}}(\mathfrak{C})$ for some hyperimaginary $\bs{e}' \in \bdd(\bs{e})$, and one of the representatives of $\bs{e}'$ is a tuple of realizations of $\Sigma$.
    Caution: Recall that $\bs{e}' \in \bdd(\bs{e})$ does not mean that $\bs{e}'$ is a countable hyperimaginary (Definition/Remark \ref{def/rem: dcl acl bdd}).
\end{enumerate} 
\end{prop}

\begin{proof}
The method of proof is the same as the proof of \cite[Proposition 2.3]{KL23}, but we use a Lascar tuple $b$ instead of a model $M$. Note that the kernel of $\pi_{\Sigma}$ is $\xi^{-1}[\autfl(\Sigma,\bs{e})]$.

$(1)\Rightarrow (2)$:
We have
$$\aut_b(\mathfrak{C}) = \aut_{b \bs{e}}(\mathfrak{C}) \le \xi^{-1}[\autfl(\Sigma, \bs{e})] \le H$$
and since $\pi_{\Sigma}[H]$ is closed, $\nu_b^{-1}[\pi_{\Sigma}[H]]$ is closed and thus $\{h(b): h \in H\}$ is type-definable over $b$.
Hence by \cite[Proposition 2.2]{KL23}, $H = \aut_{b_F \bs{e}}(\mathfrak{C})$ for some $\emptyset$-type-definable equivalence relation $F$.

We have $b_F \in \bdd(\bs{e})$:
$[\aute(\mathfrak{C}): H] = \kappa$ is small since $\autfl(\mathfrak{C}, \bs{e}) \le H$, and so there is $\{f_i \in \aute(\mathfrak{C}): i < \kappa\}$ such that $\aute(\mathfrak{C}) = \bigsqcup_{i < \kappa} f_i \cdot H$.
Then for all $g, h \in \aute(\mathfrak{C})$, if $g \cdot H = h \cdot H$, then $h^{-1} g \in H$ and hence $g(b_F) = h(b_F)$.

$(2)\Rightarrow (1)$:
Say $\bs{e}' = c_F$ where $c$ is a tuple of realizations of $\Sigma$ and $F$ is an $\emptyset$-type-definable equivalence relation.
By Remark \ref{rem:basic properties on Lascar tuples} and Remark \ref{relativized Lascar is top gp}, we may assume that the Lascar tuple $b$ contains $c$.
For $q(x) = \tp(c/\bs{e})$, because $\bs{e}' \in \bdd(\bs{e})$,
$$F'(z_1, z_2) := (q(z_1) \wedge q(z_2) \wedge F(z_1, z_2)) \vee (\neg q(z_1) \wedge \neg q(z_2))$$
is an $\bs{e}$-invariant bounded equivalence relation on $\mathfrak{C}^{|c|}$.
Since $c$ is a tuple in $\Sigma(\mathfrak{C})$, for any $\sigma \in \autfl(\Sigma, \bs{e})$, $\sigma(c) \lequiv_{\bs{e}} c$ and thus $F'(\sigma(c), c)$ by Fact \ref{fact:characterization of equivalence relations}(1).
Note that $c_F = c_{F'}$, so we have $\xi^{-1}[\autfl(\Sigma, \bs{e})] \le H = \aut_{c_F \bs{e}}(\mathfrak{C})$, hence $\pi_{\Sigma}^{-1}[\pi_{\Sigma}[H]] = H$.
Notice that $H = \{f \in \aute(\mathfrak{C}): f(c) \models F(z, c)\}$.
Then $\nu_b^{-1}[\pi_{\Sigma}[H]] = \{p(z') \in S_b(b): F(z, c) \subseteq p(z')\}$ where $z \subseteq z'$ and $|z'| = |b|$.
Thus $\pi_{\Sigma}[H]$ is closed.
\end{proof}

\begin{definition}$ $
\begin{enumerate}
    \item For $H \le \aute(\mathfrak{C})$, the relation $\equiv^H$ is the orbit equivalence relation such that for tuples $b, c$ in $\mathfrak{C}$, $b \equiv^H c$ if and only if there is $h \in H$ such that $h(b) = c$.
    \item For $H \le \aute(\Sigma)$, we will use the same notation $\equiv^H$ but it is confined to the tuples of realizations of $\Sigma$.
\end{enumerate}
\end{definition}

\begin{remark}\label{remark:orbit equiv rels}
Let $H \le \aute(\Sigma)$ and $c, d$ be tuples of realizations of $\Sigma$.
\begin{enumerate}
    \item If $H = \autfl(\Sigma, \bs{e})$, then $c \equiv^H d$ if and only if $c \lequiv_{\bs{e}} d$.
    \item If $H = \autfkp(\Sigma, \bs{e})$, then $c \equiv^H d$ if and only if $c \kpequiv_{\bs{e}} d$.
    \item If $H = \autfsh(\Sigma, \bs{e})$, then $c \equiv^H d$ if and only if $c \sequiv_{\bs{e}} d$.
\end{enumerate}
In other words, in $\Sigma$, having the same strong type is the same as the orbit equivalence relation induced from the corresponding relativized automorphism group.
From Definition \ref{def:relativized Galois groups}(3) and the definition of the orbit equivalence, this remark can be proved in a direct way.
\end{remark}

Now we prove that closed subgroups of the relativized Lascar groups are completely determined by the orbit equivalence relations they define.

\begin{lemma}\label{lemma:finite characterization of orbit equiv}
Let $H$ be a subgroup of $\aute(\Sigma)$ containing $\autfl(\Sigma, \bs{e})$ such that $\pi_b[H]$ is a closed subgroup of $\gall(\Sigma, \bs{e})$.
Then
\begin{enumerate}
    \item For any tuples $c$ and $d$ of realizations of $\Sigma$, $c \equiv^H d$ if and only if for all corresponding subtuples $c'$ and $d'$ of $c$ and $d$, which are finite tuples of realizations of $\Sigma$, $c'\equiv^H d'$.
    \item Given $\sigma\in \aute(\Sigma)$, the following are equivalent.
    \begin{enumerate}
    	\item $\sigma\in H$.
    	\item $\sigma$ fixes all the $\equiv^H$-classes of any tuples of realizations of $\Sigma$.
    	\item $\sigma$ fixes all the $\equiv^H$-classes of any finite tuples of realizations of $\Sigma$.
    \end{enumerate}
\end{enumerate}
Also, (2)(a) and (2)(b) are equivalent even when $\pi_b[H]$ is not closed.
\end{lemma}

\begin{proof}
(1): Suppose that for all corresponding subtuples $c'$ and $d'$ of $c$ and $d$, which are finite tuples of realizations of $\Sigma$, $c'\equiv^H d'$.
Note that, by Proposition \ref{prop:description of closed subgroup}, $H' := \pi_{\Sigma}^{-1}[\pi_b[H]] = \aut_{\bs{e}'\bs{e}}(\mathfrak{C})$ for some hyperimaginary $\bs{e}' \in \bdd(\bs{e})$ such that one of its representatives is a tuple of realizations of $\Sigma$.
Then by commutativity of the diagram in Remark \ref{remark:diagram of maps}, $H = \xi[H'] = \aut_{\bs{e}'\bs{e}}(\Sigma)$ and so $c \equiv^H d$ if and only if $\tp(c/\bs{e}'\bs{e}) = \tp(d/\bs{e}'\bs{e})$.
Since $\tp(c'/\bs{e}'\bs{e})=\tp(d'/\bs{e}'\bs{e})$ for all corresponding subtuples $c'$ and $d'$ of $c$ and $d$, which are finite tuples of realizations of $\Sigma$, by compactness, we have that $\tp(c/\bs{e}'\bs{e})=\tp(d/\bs{e}'\bs{e})$ and so $c \equiv^H d$.

(2): It is enough to show that $(c)$ implies $(a)$. Take arbitrary $\sigma\in \aute(\Sigma)$ and suppose $\sigma$ fixes all the $\equiv^H$-classes of any finite tuples of realizations of $\Sigma$. For any subtuple $b'$ of $b$, which is a finite tuple of realizations of $\Sigma$, $b' \equiv^H \sigma(b')$, and so by (1), $b \equiv^H \sigma(b)$.
Thus there is $\tau \in H$ such that $\tau(b) = \sigma(b)$ and so $\tau^{-1}\sigma(b) = b$.
Then we have
$$\tau^{-1} \sigma \in \aut_b(\Sigma) \le \autfl(\Sigma, \bs{e}) \le H,$$
and conclude that $\sigma \in \tau H = H$.

The last line of Lemma \ref{lemma:finite characterization of orbit equiv} can be proved in a similar way.
%
\end{proof}

\begin{definition}\label{definition: canonical sugbroups}$ $
\begin{enumerate}
    \item $\gal_{\LL}^c(\Sigma, \bs{e})$ is the topological closure of the trivial subgroup of $\gall(\Sigma, \bs{e})$.
    \item $\gal_{\LL}^0(\Sigma, \bs{e})$ is the connected component containing the identity in $\gall(\Sigma, \bs{e})$.
\end{enumerate}
\end{definition}
\noindent Note that  $\gal_{\LL}^c(\Sigma, \bs{e})$ and $\gal_{\LL}^0(\Sigma, \bs{e})$ are closed normal subgroups because $\gall(\Sigma, \bs{e})$ is a topological group.

\begin{lemma}\label{lemma:kp and sh are closed}$ $
$\autfkp(\Sigma, \bs{e}) / \autfl(\Sigma, \bs{e})$ and $\autfsh(\Sigma, \bs{e}) / \autfl(\Sigma, \bs{e})$ are closed in $\gall(\Sigma, \bs{e})$.
\end{lemma}

\begin{proof}
We have
\begin{align*}
\nu_b^{-1} [\autfkp(\Sigma, \bs{e}) / \autfl(\Sigma, \bs{e})] & = \{\tp(\sigma(b)/b): \sigma \in \autfkp(\Sigma)\}\\
& = \{\tp(c/b): c \kpequiv_{\bs{e}} b\}.
\end{align*}
By Fact \ref{fact:characterization of equivalence relations}(2) and Remark \ref{rem:type-definabilityt_equility of types}, $\kpequiv_{\bs{e}}$ is the same as $\equiv_{\bdd(\bs{e})}$, which is $\bs{e}$-invariant and type-definable.
By the assumption that $\bs{e} \in \dcl(b)$, there is $\Gamma(x, b)$ over $b$, whose solution set is the same as the solution set of the type $\tp(b/\bdd(\bs{e}))$, and thus $\{\tp(c/b): c \kpequiv_{\bs{e}} b\} = \{p(x) \in S_b(b): \models \Gamma(x, b)\}$.
Hence $\{\tp(c/b): c \kpequiv_{\bs{e}} b\}$ is closed in $S_b(b)$, and this implies that $\autfkp(\Sigma, \bs{e}) / \autfl(\Sigma, \bs{e})$ is closed.
The proof for $\autfsh(\Sigma, \bs{e}) / \autfl(\Sigma, \bs{e})$ is exactly the same;
replace KP by S and $\bdd(\bs{e})$ by $\acl(\bs{e})$.
\end{proof}

\begin{remark}\label{remark:bdd(e) cap Sigma}
Using Fact \ref{fact:countable hyperimagianry is enough}, we will relativize Definition \ref{def:dcl, acl, bdd} and Definition/Remark \ref{def/rem: dcl acl bdd} to $\Sigma$.
\begin{enumerate}
    \item If $\bs{c}$ is a hyperimaginary bounded over $\bs{e}$ and one of its representatives is a (small) tuple of realizations of $\Sigma$, then we express it by $\bs{c} \in \bdd(\bs{e}) \restriction \Sigma$.
    Note that Proposition \ref{prop:description of closed subgroup}(2) is the same as $H = \aut_{\bs{e}'\bs{e}}(\mathfrak{C})$ for some hyperimaginary $\bs{e}' \in \bdd(\bs{e}) \restriction \Sigma$.
    Define $\bs{c} \in \acl(\bs{e}) \restriction \Sigma$ similarly.
    \item For any $\bs{c} \in \bdd(\bs{e}) \restriction \Sigma$, by Fact \ref{fact:countable hyperimagianry is enough}, there is a sequence $(\bs{c}_i: i < \lambda_{\bs{c}})$ of countable hyperimaginaries, which is interdefinable with $\bs{c}$ (where $\lambda_{\bs{c}}$ is some small cardinal).
    Note that $\{\bs{c}_i: i < \lambda_{\bs{c}}\} \subseteq \bdd(\bs{e})$.
    Thus
    $$\bdd(\bs{e}) \cap \Sigma := \bigcup_{\bs{c} \in \bdd(\bs{e}) \restriction \Sigma} \{\bs{c}_i: i < \lambda_{\bs{c}}\}$$
    is a subset of $\bdd(\bs{e})$, and given $f \in \aute(\mathfrak{C})$, $f$ fixes every $\bs{c} \in \bdd(\bs{e}) \restriction \Sigma$ if and only if $f \in \aut_{\bdd(\bs{e}) \cap \Sigma}(\mathfrak{C})$.
    Define $\acl(\bs{e}) \cap \Sigma$ in the same way with $\bs{c} \in \acl(\bs{e}) \restriction \Sigma$.
\end{enumerate}

\end{remark}



\begin{remark}\label{remark: intersection is well-behaved}
Let $\{\bs{e}_i: i < \lambda\}$ be a small set of arbitrary hyperimaginaries.
According to Definition \ref{def:relativized Galois groups}(1), $\sigma \in \aut_{\bs{e}_i}(\Sigma) = \xi[\aut_{\bs{e}_i}(\mathfrak{C})]$ does {\em not} mean that $\sigma$ fixes $\bs{e}_i$ (because $\bs{e}_i \subseteq \Sigma(\mathfrak{C})$ is not assumed), and it is not clear whether
$$\bigcap_{i < \lambda} \aut_{\bs{e}_i}(\Sigma) = \aut_{\{\bs{e}_i: i < \lambda\}}(\Sigma) \ \ (\dagger)$$
holds or not (note that $\bigcap_{i < \lambda} \aut_{\bs{e}_i}(\Sigma) = \bigcap_{i<\lambda} \xi[\aut_{\bs{e}_i}(\mathfrak{C})] \supseteq \xi[\bigcap_{i<\lambda}\aut_{\bs{e}_i}(\mathfrak{C})] = \aut_{\{\bs{e}_i: i < \lambda\}}(\Sigma)$, thus one direction of the containment is trivial).

But in the case where $\bs{e}_i$ is a hyperimaginary such that one of the representatives of $\bs{e}_i$ is a tuple of realizations of $\Sigma$ for each $i < \lambda$, $(\dagger)$ holds.
\end{remark}

\begin{proof}
Say $a_i$ is one representative of $\bs{e}_i$ in $\Sigma(\mathfrak{C})$.
If $\sigma \in \bigcap_{i < \lambda} \aut_{\bs{e}_i}(\Sigma)$, then for each $i < \lambda$, $\sigma(a_i)$ is still a representative of $\bs{e}_i$ in $\Sigma(\mathfrak{C})$, thus any extension of $\sigma$ should fix $\bs{e}_i$.
That is, any extension $f \in \aute(\mathfrak{C})$ of $\sigma$ fixes every hyperimaginary in $\{\bs{e}_i: i < \lambda\}$, hence $\sigma \in \aut_{\{\bs{e}_i: i < \lambda\}}(\Sigma)$.
\end{proof}

\begin{lemma}\label{lemma:preimages are bdd and acl}
$$\pi_b^{-1}[\gal^c_{\LL}(\Sigma, \bs{e})] = \aut_{\bdd(\bs{e}) \cap \Sigma}(\Sigma)\text{, and}$$
$$\pi_b^{-1}[\gal^0_{\LL}(\Sigma, \bs{e})] = \aut_{\acl(\bs{e}) \cap \Sigma}(\Sigma).$$
\end{lemma}

\begin{proof}
By the definition, $\gal^c_{\LL}(\Sigma, \bs{e})$ is the intersection of all closed subgroups of $\gall(\Sigma, \bs{e})$.
But by Proposition \ref{prop:description of closed subgroup}, the collection of all closed subgroups of $\gall(\Sigma, \bs{e})$ is the same as $\{\pi_{\Sigma}(\aut_{\bs{e}'\bs{e}}(\mathfrak{C})): \bs{e}' \in \bdd(\bs{e}) \restriction \Sigma\}$.
Thus we have
$$\gal^c_{\LL}(\Sigma, \bs{e}) = \bigcap_{\bs{e}' \in \bdd(\bs{e}) \restriction \Sigma} \pi_{\Sigma}[\aut_{\bs{e}'\bs{e}}(\mathfrak{C})].$$

\begin{claim}\label{claim:having the same Lascar in Sigma}
Let $\sigma \in \autfl(\Sigma, \bs{e})$ and $\bs{e}' \in \bdd(\bs{e}) \restriction \Sigma$.
Then for any extension $f \in \aute(\mathfrak{C})$ of $\sigma$, we have $f \in \aut_{\bs{e}'\bs{e}}(\mathfrak{C})$, which implies that $\sigma \in \aut_{\bs{e}' \bs{e}}(\Sigma)$.
In particular, it follows that $\autfl(\Sigma, \bs{e}) \le \aut_{\bs{e}'\bs{e}}(\Sigma)$.
\end{claim}

\begin{proof}[Proof of the claim]
Let $a' \in \Sigma(\mathfrak{C})$ be a representative of $\bs{e}'$ and $f \in \aute(\mathfrak{C})$ be any extension of $\sigma$.
Then $f(a') = \sigma(a') \lequiv_{\bs{e}} a'$, thus $f(\bs{e}') \lequiv_{\bs{e}} \bs{e}'$ and so $f(\bs{e}') \equiv_{\bdd(\bs{e})} \bs{e}'$ by Fact \ref{fact:characterization of equivalence relations}.
But $\bs{e}' \in \bdd(\bs{e})$, hence $f(\bs{e}') = \bs{e}'$.
\end{proof}

Therefore
\begin{align*}
\pi_b^{-1}[\gal^c_{\LL}(\Sigma, \bs{e})] & = \bigcap_{\bs{e}' \in \bdd(\bs{e}) \restriction \Sigma} \pi_b^{-1} [\pi_{\Sigma}[\aut_{\bs{e}'\bs{e}}(\mathfrak{C})]]\\
& = \bigcap_{\bs{e}' \in \bdd(\bs{e}) \restriction \Sigma} \xi[\aut_{\bs{e}'\bs{e}}(\mathfrak{C})]\\
& = \bigcap_{\bs{e}' \in \bdd(\bs{e}) \restriction \Sigma} \aut_{\bs{e}'\bs{e}}(\Sigma)\\
& = \aut_{\bdd(\bs{e}) \cap \Sigma}(\Sigma),
\end{align*}
where the second equality follows because
$$\pi_b^{-1} [\pi_{\Sigma}[\aut_{\bs{e}'\bs{e}}(\mathfrak{C})]] = \pi_b^{-1} [\pi_b \circ \xi[\aut_{\bs{e}'\bs{e}}(\mathfrak{C})]] = (\pi_b^{-1} \circ \pi_b) [\aut_{\bs{e}'\bs{e}}(\Sigma)]$$
and we have Claim \ref{claim:having the same Lascar in Sigma};
the last equality follows by Remark \ref{remark:bdd(e) cap Sigma} and \ref{remark: intersection is well-behaved}.
We have proved the first statement of this lemma.

For the second statement of this lemma, we prove a claim first:

\begin{claim}\label{claim:conn component is intersection}
$$\gal^0_{\LL}(\Sigma, \bs{e}) = \bigcap_{\bs{e}' \in \acl(\bs{e}) \restriction \Sigma} \pi_{\Sigma}[\aut_{\bs{e}' \bs{e}}(\mathfrak{C})].$$
\end{claim}

\begin{proof}[Proof of the claim]
The proof of this claim is essentially the same as the first part of the proof of \cite[Proposition 3.12(1)]{KL23}.
Recall the definition of $\gal^0_{\LL}(\Sigma, \bs{e})$ and the classical fact that the connected component containing the identity is the intersection of all closed (normal) subgroups of finite indices in a topological group.
Note that for any $\bs{e}' \in \acl(\bs{e}) \restriction \Sigma$, $\aut_{\bs{e}' \bs{e}}(\mathfrak{C})$ has a finite index in $\aute(\mathfrak{C})$, and hence $\pi_{\Sigma}[\aut_{\bs{e}' \bs{e}}(\mathfrak{C})]$ has a finite index in $\pi_{\Sigma}[\aute(\mathfrak{C})] = \gall(\Sigma, \bs{e})$;
$\pi_{\Sigma}[\aut_{\bs{e}' \bs{e}}(\mathfrak{C})]$ is also closed due to Proposition \ref{prop:description of closed subgroup}.
This proves that
$$\gal^0_{\LL}(\Sigma, \bs{e}) \subseteq \bigcap_{\bs{e}' \in \acl(\bs{e}) \restriction \Sigma} \pi_{\Sigma}[\aut_{\bs{e}' \bs{e}}(\mathfrak{C})].$$

On the other hand, by Proposition \ref{prop:description of closed subgroup}, every closed subgroup of $\gall(\Sigma, \bs{e})$ should be $\pi_{\Sigma}[\aut_{\bs{e}'\bs{e}}(\mathfrak{C})]$ for some $\bs{e}' \in \bdd(\bs{e}) \restriction \Sigma$.
If this $\pi_{\Sigma}[\aut_{\bs{e}'\bs{e}}(\mathfrak{C})]$ has a finite index in $\gall(\Sigma, \bs{e})$,
then $\aut_{\bs{e}'\bs{e}}(\mathfrak{C})$ has a finite index in $\pi_{\Sigma}^{-1}(\gall(\Sigma, \bs{e})) = \aute(\mathfrak{C})$, thus $\bs{e}' \in \acl(\bs{e}) \restriction \Sigma$.
This proves that
$$\gal^0_{\LL}(\Sigma, \bs{e}) \supseteq \bigcap_{\bs{e}' \in \acl(\bs{e}) \restriction \Sigma} \pi_{\Sigma}[\aut_{\bs{e}' \bs{e}}(\mathfrak{C})].$$
\end{proof}

By Claim \ref{claim:conn component is intersection},
\begin{align*}
\pi_b^{-1}[\gal^0_{\LL}(\Sigma, \bs{e})] & = \bigcap_{\bs{e}' \in \acl(\bs{e}) \restriction \Sigma} \pi_b^{-1} [\pi_{\Sigma}[\aut_{\bs{e}'\bs{e}}(\mathfrak{C})]]\\
& = \bigcap_{\bs{e}' \in \acl(\bs{e}) \restriction \Sigma} \xi[\aut_{\bs{e}'\bs{e}}(\mathfrak{C})]\\
& = \bigcap_{\bs{e}' \in \acl(\bs{e}) \restriction \Sigma} \aut_{\bs{e}'\bs{e}}(\Sigma)\\
& = \aut_{\acl(\bs{e}) \cap \Sigma}(\Sigma),
\end{align*}
where the second equality follows from Claim \ref{claim:having the same Lascar in Sigma} (as $\bs{e}' \in \acl(\bs{e}) \restriction \Sigma$ implies $\bs{e}' \in \bdd(\bs{e}) \restriction \Sigma$) with the same argument as before;
the last equality follows by Remark \ref{remark:bdd(e) cap Sigma} and \ref{remark: intersection is well-behaved}.

\end{proof}

\begin{proposition}\label{prop:id clousre, component finest} Let $c$ and $d$ be tuples of realizations of $\Sigma$.
\begin{enumerate}
    \item Let $H = \pi_b^{-1}[\gal_{\LL}^c(\Sigma, \bs{e})] \le \aute(\Sigma)$.
    Then $c \equiv^H d$ if and only if $c \kpequiv_{\bs{e}} d$.
    \item Let $H = \pi_b^{-1}[\gal_{\LL}^0(\Sigma, \bs{e})] \le \aute(\Sigma)$.
    Then $c \equiv^H d$ if and only if $c \sequiv_{\bs{e}} d$.
\end{enumerate}
\end{proposition}

\begin{proof}
(1): By Proposition \ref{prop:description of closed subgroup},
we know that $\pi_{\Sigma}^{-1}[\gal_{\LL}^c(\Sigma, \bs{e})] = \aut_{\bs{e}'\bs{e}}(\mathfrak{C})$ for some $\bs{e}' \in \bdd(\bs{e})$.
Thus by commutativity of the diagram in Remark \ref{remark:diagram of maps}, $H = \xi[\aut_{\bs{e}'\bs{e}}(\mathfrak{C})] = \aut_{\bs{e}'\bs{e}}(\Sigma)$ and so $c \equiv^{H} d$ if and only if $c \equiv_{\bs{e}'\bs{e}} d$.
By Fact \ref{fact:characterization of equivalence relations}(2), $c \kpequiv_{\bs{e}} d$ if and only if $c \equiv_{\bdd(\bs{e})} d$, and $\bs{e}' \in \bdd(\bs{e})$, thus $c \kpequiv_{\bs{e}} d$ implies $c \equiv^H d$.

By Lemma \ref{lemma:kp and sh are closed}, we have $\gal_{\LL}^c(\Sigma, \bs{e}) \le \autfkp(\Sigma, \bs{e}) / \autfl(\Sigma, \bs{e})$ and hence $H\le \autfkp(\Sigma,\bs{e})$.
We already have proved that $c \kpequiv_{\bs{e}} d$ implies $c \equiv^H d$, and by Remark \ref{remark:orbit equiv rels}, $c \equiv^{\autfkp(\Sigma, \bs{e})} d$ if and only if $c \kpequiv_{\bs{e}} d$, so it follows that $c \equiv^H d$ if and only if $c \kpequiv_{\bs{e}} d$.

(2): By Fact \ref{fact:characterization of equivalence relations}(3), if $c \sequiv_{\bs{e}} d$, then $c \equiv_{\acl(\bs{e})} d$, so by Lemma \ref{lemma:preimages are bdd and acl}, $c \equiv^H d$.
Conversely, assume $c \equiv^H d$.
By Fact \ref{fact:characterization of equivalence relations}(3), it suffices to show that if $F$ is an $\bs{e}$-invariant type-definable equivalence relation and $c_F$ has finitely many automorphic images over $\bs{e}$, then $F(c, d)$ holds.


Since $F$ is $\bs{e}$-invariant and $\bs{e} \in \dcl(b)$, by Fact \ref{fact:e-hi is empty hi}, there is an $\emptyset$-type-definable equivalence relation $F'$ such that $c_F$ and $(cb)_{F'}$ are interdefinable over $\bs{e}$.
But then $(cb)_{F'}$ is a hyperimaginary in $\acl(\bs{e})$ and $cb$ is a tuple of realizations of $\Sigma$, thus $(cb)_{F'} \in \acl(\bs{e}) \restriction \Sigma$.
By Lemma \ref{lemma:preimages are bdd and acl}, we know that there is $f \in \aut_{\acl(\bs{e}) \cap \Sigma}(\mathfrak{C})$ such that $f(c) = d$.
Then $f \in \aut_{(cb)_{F'}}(\mathfrak{C})$ by Remark \ref{remark:bdd(e) cap Sigma}, and $\aut_{(cb)_{F'}}(\mathfrak{C}) = \aut_{c_F}(\mathfrak{C})$ and so $F(c,d)$.
\end{proof}

Now we can give an answer to one of the main questions of this section, which says that the quotient groups of relativized KP and Shelah automorphisms are exactly the canonical normal subgroups of the relativized Lascar group.
Recall that $\bs{e} \in \dcl(\Sigma)$ means that there is a small tuple $c$ of realizations of $\Sigma$ such that $\bs{e} \in \dcl(c)$.

\begin{theorem}\label{thm:description of KP type}
Let $\Sigma$ be an $\bs{e}$-invariant type and suppose $\bs{e} \in \dcl(\Sigma)$.
Then
\begin{enumerate}
    \item $\gal_{\LL}^c(\Sigma, \bs{e}) = \autfkp(\Sigma, \bs{e})/\autfl(\Sigma, \bs{e})$.
    \item $\gal_{\LL}^0(\Sigma, \bs{e}) = \autfsh(\Sigma, \bs{e})/\autfl(\Sigma, \bs{e})$.
\end{enumerate}
\end{theorem}

\begin{proof}
(1): Let $H_1 = \pi_b^{-1}[\gal_{\LL}^c(\Sigma, \bs{e})]$ and $H_2 = \autfkp(\Sigma, \bs{e})$.
By Remark \ref{remark:orbit equiv rels} and Proposition \ref{prop:id clousre, component finest}, we have that $\equiv^{H_1}$ and $\equiv^{H_2}$ are the same equivalence relations on the tuples of realizations of $\Sigma(\mathfrak{C})$.
So, by Lemma \ref{lemma:finite characterization of orbit equiv}, $H_1 = H_2$.

(2): By exactly the same proof as (1), letting $H_1 = \pi_b^{-1}[\gal_{\LL}^0(\Sigma, \bs{e})]$ and $H_2 = \autfsh(\Sigma, \bs{e})$.
\end{proof}

We summarize the results obtained to prove Theorem \ref{thm:description of KP type} in this section, and state the characterization of strong types in $\Sigma$, which is analogous to Fact \ref{fact:characterization of equivalence relations}, in the following corollary.
It says that having the same strong type can be, in some sense, localized by finding a partial type that contains the real tuples of interest.

\begin{corollary}\label{cor:char of KP and SH in Sigma}
Let $\Sigma$ be an $\bs{e}$-invariant type and suppose $\bs{e} \in \dcl(\Sigma)$.
Let $c$ and $d$ be tuples of realizations of $\Sigma$.
\begin{enumerate}
    \item The following are equivalent.
    \begin{enumerate}
        \item $c \kpequiv_{\bs{e}} d$.
        \item $c \equiv_{\bdd(\bs{e}) \cap \Sigma} d$.
        \item There is $\sigma \in \pi_b^{-1}[\gal_{\LL}^c(\Sigma, \bs{e})]$ such that $\sigma(c) = d$.
    \end{enumerate}
    \item The following are equivalent.
    \begin{enumerate}
        \item $c \shequiv_{\bs{e}} d$
        \item $c \equiv_{\acl(\bs{e}) \cap \Sigma} d$.
        \item There is $\sigma \in \pi_b^{-1}[\gal_{\LL}^0(\Sigma, \bs{e})]$ such that $\sigma(c) = d$.
    \end{enumerate}
\end{enumerate}
\end{corollary}

\begin{proof}
(1): $(a) \Rightarrow (b)$ is clear by Fact \ref{fact:characterization of equivalence relations}(2).
By Lemma \ref{lemma:preimages are bdd and acl}, $\pi_b^{-1}[\gal_{\LL}^c(\Sigma, \bs{e})] = \aut_{\bdd(\bs{e}) \cap \Sigma}(\Sigma)$, so $(b)$ and $(c)$ are equivalent.
$(c) \Rightarrow (a)$ follows from Proposition \ref{prop:id clousre, component finest}.
The proof for (2) is similar.
\end{proof}

\section{Restricted model theoretic Galois groups}\label{sec:rest and example}

There is another possible approach to consider the subgroups of the automorphism group in $\Sigma$, by restricting the domains of the automorphisms in the subgroup, which was considered in \cite[Section 3]{DKL17}.

\begin{definition}\label{def:restricted groups}
For $X \in \{\LL, \KP, \SH\}$,
$$\autf_X^{\res}(\Sigma, \bs{e}) = \{f \upharpoonright \Sigma(\mathfrak{C}): f \in \autf_X(\mathfrak{C}, \bs{e})\} \le \aute(\Sigma).$$
\end{definition}

\begin{remark}\label{remark:orbit_equiv_are_same}$ $
\begin{enumerate}
    \item For $X \in \{\LL, \KP, \SH\}$, $\autf_X^{\res}(\Sigma, \bs{e}) \le \autf_X(\Sigma, \bs{e})$.
    \item Let $H_1 = \autf_X^{\res}(\Sigma, \bs{e})$ and $H_2 = \autf_X(\Sigma, \bs{e})$.
    Then $\equiv^{H_1}$ and $\equiv^{H_2}$ are the same.
\end{enumerate}

\end{remark}
\begin{proof}
(1) is clear, and for (2), we only need to show that $c \equiv^{H_2} d$ implies $c \equiv^{H_1} d$ where $c, d$ are tuples of realizations of $\Sigma$.
By Remark \ref{remark:orbit equiv rels}, $\equiv^{H_2}$ is the same as $\equiv^X_{\bs{e}}$, thus $c \equiv^{H_2} d$ implies $c \equiv^X_{\bs{e}} d$ and hence there is $f \in \autf_X(\mathfrak{C}, \bs{e})$ such that $f(c) = d$.
So, $c \equiv^{H_1} d$.
\end{proof}

Of course, we can relativize or take a restriction of any subgroup of the automorphism group.
We will use the following notation for general subgroups.

\begin{definition}\label{def:general definition of subgroup}
Given a subgroup  $H \le \aute(\mathfrak{C})$, let $\equiv^H$ be the orbit equivalence relation.
\begin{enumerate}
    \item $\aut_{\equiv^H}(\Sigma, \bs{e}) = \{\sigma \in \aute(\Sigma): $ for any tuple $c$ of realizations of $\Sigma$, $c \equiv^H \sigma(c)\}$.
    \item $H^{\res} = \xi[H] =  \{f \upharpoonright \Sigma(\mathfrak{C}): f \in H\}$.
\end{enumerate}
\end{definition}

\begin{remark}\label{remark:res for L, KP, Sh}$ $
\begin{enumerate}
    \item For any $H \le \aute(\mathfrak{C})$, $H^{\res} \le \aut_{\equiv^H}(\Sigma, \bs{e})$.
    \item For $X \in \{\LL, \KP, \SH\}$,
    \begin{enumerate}
        \item $\autf_X(\Sigma, \bs{e}) = \aut_{\equiv^{X}_{\bs{e}}}(\Sigma, \bs{e})$, and
        \item $\autf_X^{\res}(\Sigma, \bs{e}) = \autf_X(\mathfrak{C}, \bs{e})^{\res}$.
    \end{enumerate}
\end{enumerate}
\end{remark}

The following proposition implies that if $\autfl(\mathfrak{C}, \bs{e}) \le H$ and $\autfl(\Sigma, \bs{e}) \le H^{\res}$, then $H^{\res} = \aut_{\equiv^H}(\Sigma, \bs{e})$.

\begin{proposition}\label{prop:orbit equivlance and res}
Let $H \le \aute(\mathfrak{C})$ such that $\equiv^H$ is weaker than $\lequiv_{\bs{e}}$ (that is, if $c \lequiv_{\bs{e}} d$, then $c \equiv^H d$).
Then
$$H^{\res} \autfl(\Sigma, \bs{e}) = \aut_{\equiv^H}(\Sigma, \bs{e})$$
where $H^{\res} \autfl(\Sigma, \bs{e}) = \{g_1g_2 \in \aute(\Sigma): g_1 \in H^{\res}$, $g_2 \in \autfl(\Sigma, \bs{e})\}$.
Especially, for $X \in \{\LL, \KP, \SH\}$, $\autf_X^{\res}(\Sigma, \bs{e}) \autfl(\Sigma, \bs{e}) = \autf_X(\Sigma, \bs{e})$.
\end{proposition}

\begin{proof}
$(\subseteq)$: Let $\sigma \in H^{\res} \autfl(\Sigma, \bs{e})$, so that $\sigma = \tau \tau'$ for some $\tau \in H^{\res}$ and $\tau' \in \autfl(\Sigma, \bs{e})$.
It is enough to show that for a Lascar tuple $b$, $\sigma(b) \equiv^H b$, and
$$\sigma(b) = \tau \tau'(b) \equiv^H \tau'(b) \equiv^H b.$$

$(\supseteq)$: Take $\sigma \in \aut_{\equiv^H}(\Sigma, \bs{e})$.
For a Lascar tuple $b$ in $\Sigma$, $b \equiv^H \sigma(b)$, thus there is $\tau \in H$ such that $\tau(b) = \sigma(b)$.
Then letting $\tau_{\Sigma} = \tau \upharpoonright \Sigma(\mathfrak{C})$, we have
$$\tau_{\Sigma}^{-1} \sigma \in \aut_b(\Sigma) \le \autfl(\Sigma, \bs{e}),$$ and so $$\sigma \in \tau_{\Sigma} \cdot \autfl(\Sigma, \bs{e}) \subseteq H^{\res} \autfl(\Sigma, \bs{e}).$$
\end{proof}


Consideration of the small box topology in Remark \ref{rem:res_aut_ptwise_top}, and Example \ref{example:necessary condition on H} are suggested by the referee.

\begin{remark}\label{rem:res_aut_ptwise_top}
Let $H \le \aute(\mathfrak{C)}$ and consider the small box topology on $\aute(\Sigma)$, that is, let basic open sets be of the form $U_{c,d}:=\{\sigma\in \aute(\Sigma):\sigma(c)=d\}$, where $c$ and $d$ are small tuples of realizations of $\Sigma$.

\begin{itemize}
    \item $\aut_{\equiv^H}(\Sigma,\bs{e})$ is a closed subset of $\aute(\Sigma)$:\\
    Given $\sigma \notin \aut_{\equiv^H}(\Sigma,\bs{e})$, there is a small tuple $c$ of realizations of $\Sigma$ such that $c \not\equiv^H \sigma(c)$.
    Then the open neighborhood $U_{c,\sigma(c)}$ of $\sigma$ is disjoint from $\aut_{\equiv^H}(\Sigma,\bs{e})$.
    \item $H^{\res}$ is dense in $\aut_{\equiv^H}(\Sigma,\bs{e})$:\\
    If $U_{c, d}$ intersects $\aut_{\equiv^H}(\Sigma,\bs e)$, then $c \equiv^H d$ and so $U_{c, d}$ intersects $H^{\res}$.
\end{itemize}
Thus we can conclude that $\aut_{\equiv^H}(\Sigma,\bs{e})$ is the closure of $H^{\res}$ in the small box topology.

If $\xi[H]( = H^{\res})$ contains $\autfl(\Sigma, \bs{e})$ and $\pi_{\Sigma}[H]$ is a closed subgroup of $\gall(\Sigma, \bs{e})$, then the pointwise convergence topology on $\aute(\Sigma)$ with basic open sets $U_{c, d} = \{\sigma \in \aute(\Sigma): \sigma(c) = d\}$, where $c$ and $d$ are finite tuples of realizations of $\Sigma$ gives the same result:

In the first bullet of the argument for the small box topology, by Lemma \ref{lemma:finite characterization of orbit equiv}(1), $\sigma \notin \aut_{\equiv^H}(\Sigma, \bs{e})$ implies that there is a finite tuple $c$ of realizations of $\Sigma$ such that $c \not\equiv^H \sigma(c)$ (note that $\equiv^H$ is the same as $\equiv^{H^{\res}}$ for tuples of realizations of $\Sigma$).
Thus $\aut_{\equiv^H}(\Sigma,\bs{e})$ is also the closure of $H^{\res}$ in the pointwise convergence topology. 


\end{remark}

According to the observations in Remark \ref{rem:res_aut_ptwise_top}, we can easily find an example such that $H^{\res} \ne \aut_{\equiv^H}(\Sigma, \bs{e})$:

\begin{example}\label{example:necessary condition on H}
Let $\mathfrak{C}$ be a saturated pure set in the empty language, $\bs{e}=\emptyset$, and $\Sigma=\{x=x\}$.
Note that $\aut(\mathfrak{C})=\autfl(\mathfrak{C})$ is the group of all permutations of $\mathfrak{C}$.

Now, let $H \le \aut(\mathfrak{C})$ be the group of all permutations of small support.

Then $\equiv^H$ holds between every two small tuples of $\mathfrak{C}$, and $H^{\res} = H\neq \aut(\mathfrak{C})=\aut_{\equiv^H}(\Sigma, \bs{e})$.

\end{example}

Note that in Example \ref{example:necessary condition on H}, $\autflres(\Sigma, \bs{e}) = \aut(\mathfrak{C}) = \autfl(\Sigma, \bs{e})$, thus this example does not answer the second question of \cite[Question 3.8]{DKL17}, which asked if there is an example such that $\autflres(\Sigma, \emptyset) \ne \autfl(\Sigma, \emptyset)$.
Also, in Example \ref{example:necessary condition on H}, $\gall(T, \bs{e}) = \gall(\Sigma, \bs{e})$ is trivial (so that $\pi[H]$ and $\aut_{\equiv^H}(\Sigma, \bs{e})$ are trivially closed), but $H = H^{\res}$ does not contain $\autfl(\mathfrak{C}, \bs{e}) = \autfl(\Sigma, \bs{e}) = \aut(\mathfrak{C})$. Based on Remark \ref{rem:res_aut_ptwise_top} and Example \ref{example:necessary condition on H}, we ask the following question, which is a variant of the second question of \cite[Question 3.8]{DKL17}.


\begin{question}\label{question:res=fix} $ $
\begin{enumerate}
	\item Do we have that $\autf_{\LL}^{\res}(\Sigma,\bs{e})=\autfl(\Sigma,\bs{e})$? By Proposition \ref{prop:orbit equivlance and res}, this implies that for any subgroup $H$ of $\aute(\mathfrak{C})$ containing $\autfl(\mathfrak{C},\bs{e})$, $H^{\res}=\aut_{\equiv^H}(\Sigma,\bs{e})$.
	\item For a subgroup $H$ of $\aute(\mathfrak{C})$ containing $\autfl(\mathfrak{C},\bs{e})$ such that $\pi[H]$ is a closed subgroup of $\gall(T,\bs{e})$ and $\aut_{\equiv^H}(\Sigma,\bs{e})$ is a closed subgroup of $\gall(\Sigma,\bs{e})$, do we have that $H^{\res}= \aut_{\equiv^H}(\Sigma,\bs{e})$? Especially, do we have that $\autfkp(\mathfrak{C},\bs{e})^{\res}=\autfkp(\Sigma,\bs{e})$ and $\autfsh(\mathfrak{C},\bs{e})^{\res}=\autfsh(\Sigma,\bs{e})$?
\end{enumerate}

\end{question}



\end{document}